\documentclass[12pt]{amsart}

\usepackage{amsthm, amsmath}
\usepackage{amssymb}
\usepackage[all]{xy}
\usepackage{mathrsfs}
\usepackage[latin1]{inputenc}
\usepackage{graphicx}
\usepackage{xspace}
\usepackage{enumerate}

\numberwithin{equation}{subsection}

\DeclareMathOperator{\Gal}{Gal}

\DeclareMathOperator{\Adm}{Adm}

\theoremstyle{plain}

\newtheorem{theorem}{Theorem}[section]

\newtheorem*{thrm1}{Theorem A}

\newtheorem*{thrm2}{Theorem B}

\newtheorem{thm}[theorem]{Theorem}
\newtheorem{cor}[theorem]{Corollary}

\newtheorem{prop}[theorem]{Proposition}

\theoremstyle{definition}
\newtheorem{rem}[theorem]{Remark}

\def\ge{\geqslant}
\def\le{\leqslant}
\def\a{\alpha}
\def\b{\beta}
\def\g{\gamma}
\def\G{\Gamma}

\def\e{\epsilon}

\def\s{\sigma}
\def\t{\tau}

\def\k{\kappa}
\def\l{\lambda}

\def\i{^{-1}}
\def\tSS{\tilde{\mathbb S}}
\def\SS{\mathbb S}

\def\NN{\mathbb N}
\def\QQ{\mathbb Q}
\def\RR{\mathbb R}

\def\ca{\mathcal A}

\def\cg{\mathcal G}

\def\cl{\mathcal L}

\def\co{\mathcal O}

\def\aa{\mathbf a}

\def\tW{\tilde W}

\usepackage{amssymb} % symboler fra AMS
\usepackage{latexsym} % LaTeX symboler
\usepackage{amsfonts} % fonte fra AMS
\usepackage{amsmath} % matematik fonte fra AMS
\usepackage{eucal} % 'krlle'-fonte (skriv \mathcal)
\usepackage{bm} % fed skrift i matematik
\usepackage{bbm} % skrift med dobbelt streg til f.eks. talomr�er
\usepackage{graphicx} % mulighed for at inkludere grafik
\usepackage[english]{varioref} % smarte referencer - se nederst
\usepackage[nice]{nicefrac} % p�ere brker
\usepackage[all]{xy}

\newcommand{\kk}{\Bbbk}

\newcommand{\FF}{\mathbbm{F}}

\def\<{\langle}
\def\>{\rangle}

\begin{document}

\title[Kottwitz-Rapoport conjecture]{Kottwitz-Rapoport conjecture on unions of affine Deligne-Lusztig varieties\\[.6cm]La conjecture de Kottwitz-Rapoport sur les unions de vari\'et\'es de Deligne-Lusztig affines}

\author[X. He]{Xuhua He}
\address{Department of Mathematics, University of Maryland, College Park, MD 20742, USA and department of Mathematics, HKUST, Hong Kong}
\thanks{X. He was partially supported by Hong Kong RGC grant 602011.}
\email{xuhuahe@math.umd.edu}

\begin{abstract}
In this paper, we prove a conjecture of Kottwitz and Rapoport on a union of (generalized) affine Deligne-Lusztig varieties $X(\mu, b)_J$ for any tamely ramified group $G$ and its parahoric subgroup $P_J$. We show that $X(\mu, b)_J \neq \emptyset$ if and only if the group-theoretic version of Mazur's inequality is satisfied. In the process, we obtain a generalization of Grothendieck's conjecture on the closure relation of $\s$-conjugacy classes of a twisted loop group. 
\vskip1cm
Dans cet article nous prouvons une conjecture de Kottwitz et Rapoport sur l'union de vari\'et\'es de Deligne-Lusztig affines (g\'en\'eralis\'ees) $X(\mu,b)_J$ pour $G$ un groupe mod\'er\'ement ramifi\'e et $P_J$ son sous-groupe parahorique. Nous montront que $X(\mu, b)_J$ est non vide si et seulement si la version de l'in\'egalit\'e de Mazur pour les groupes est satisfaite. Au cours de la preuve, nous obtenons une g\'en\'eralisation de la conjecture de Grothendieck sur les inclusions des adh\'erences de classes de $\s$-conjugaison d'un groupe de lacets tordu.
\end{abstract}

\maketitle

\section*{Introduction}

\subsection{} The motivation of this paper comes from the reduction of Shimura varieties with a parahoric level structure. On the special fiber, there are two important stratifications:

\begin{itemize}

\item Newton stratification, indexed by specific $\s$-conjugacy classes in the associated $p$-adic group $G$. 

\item Kottwitz-Rapoport stratification, indexed by specific double cosets in $W_J \backslash \tW/W_J$, where $\tW$ is the Iwahori-Weyl group of $G$ and $W_J$ is the Weyl group of the parahoric subgroup $P_J$. 
\end{itemize}

A fundamental question is to determine which Kottwitz-Rapoport strata and which Newton strata are nonempty, in other words, to determine the double cosets of $W_J \backslash \tW/W_J$ and the subset of $\s$-conjugacy classes that appear in the reduction of Shimura varieties. 

It consists of two parts: local theory and global theory. In this paper, we focus on local theory. 

\subsection{} In \cite{Z} and \cite{PZ}, Pappas and Zhu give a group-theoretic definition of ``local models'' of Shimura varieties and show that the subset of $W_J \backslash \tW/W_J$ for the local model is the admissible set $\Adm_J(\mu)$ (defined in \S\ref{1.5}). 

The next question is to describe the $\s$-conjugacy classes arises in the reduction of Shimura varieties. Based on some foundational relations between Newton strata, Kottwitz-Rapoport strata and affine Deligne-Lusztig varieties, we study the set $X(\mu, b)_J$, a union of generalized affine Deligne-Lusztig varieties indexed by $\Adm_J(\mu)$. It is defined as follows. Let $L$ be the completion of the maximal unramified extension of a $p$-adic field and $b \in G(L)$, set
$$X(\mu, b)_J=\{g P_J \in G(L)/P_J; g \i b \s(g) \in \cup_{w \in \Adm_J(\mu)} P_J w P_J\}.$$ 

Kottwitz and Rapoport introduced a set $B(G, \mu)$ of acceptable $\s$-conjugacy classes, defined by the group-theoretic version of Mazur's theorem. The main purpose of this paper is to prove the following result, conjectured by Kottwitz and Rapoport in \cite{KR} and \cite{R:guide}. 

\begin{thrm1}
$X(\mu, b)_J \neq \emptyset$ if and only if $[b] \in B(G, \mu)$. 
\end{thrm1}

%\begin{rem} This is a group-theoretic statement and we don't require that $(G, \{\mu\})$ is a Shimura datum. Thus the result is also applied to G-Shtukas. 
%\end{rem}

\subsection{} The direction $$X(\mu, b)_J \neq \emptyset \Rightarrow [b] \in B(G, \mu)$$ is the group-theoretic version of Mazur's inequality between the Hodge polygon of an F-crystal and the Newton polygon of its underlying F-isocrystal. The case where $G$ is an unramified group and $P_J$ is a hyperspecial maximal subgroup, is proved by Rapoport and Richartz in \cite[Theorem 4.2]{RR}. Another proof is given by Kottwitz in \cite{Ko3}. The case where $G$ is an unramified group and $P_J$ is an Iwahori subgroup, is proved in \cite[Notes added June 2003, (7)]{R:guide}.

The other direction $$X(\mu, b)_J \neq \emptyset \Leftarrow [b] \in B(G, \mu)$$ is the ``converse to Mazur's inequality'' and was proved by Wintenberger in \cite{Wi} in case $G$ is quasi-split. 

\subsection{} Another related question is to determine the non-emptiness pattern for a single affine Deligne-Lusztig variety. 

If $G$ is quasi-split and $P_J$ is a special maximal parahoric subgroup, then the non-emptiness pattern of a single affine Deligne-Lusztig variety is still governed by Mazur's inequality. It is conjectured and proved for $G=GL_n$ or $GSp_{2n}$ by Kottwitz and Rapoport in \cite{KR}. It is then proved by Lucarelli \cite{Luc} for classical split groups and then by Gashi \cite{Ga} for unramified cases. The general case is proved in \cite[Theorem 7.1]{He99}. Notice that if $P_J$ is a special maximal parahoric subgroup and $\mu$ is minuscule with respect to $\tW$, $X(\mu, b)_J$ is in fact a single affine Deligne-Lusztig variety. 

If $P_J$ is an Iwahori subgroup and $b$ is basic, a conjecture on the non-emptiness pattern (for split groups) is given by G\"ortz, Haines, Kottwitz, and Reuman in \cite{GHKR} in terms of $P$-alcoves in \cite{GHKR} and the generalization of this conjecture to any tamely ramified groups is proved in \cite{GHN}. The non-emptiness pattern for basic $b$ and other parahoric subgroups can then be deduced from Iwahori case easily. 

However, such information is not useful for the study of $X(\mu, b)_J$. The reason is that for $b$ basic, it is very easy to determine whether $X(\mu, b)_J$ is empty (by checking the image under Kottwitz map) and for other $b$, and non-special parahoric subgroup $J$, very little is known about the non-emptiness pattern for a single affine Deligne-Lusztig variety. 

\subsection{} Now we discuss the strategy of the proof of Theorem A. The key ingredients are 
\begin{itemize} 
\item the partial order on $B(G)$; 

\item some nice properties on the admissible set $\Adm_J(\mu)$;

\item the fact that the maximal element in $B(G, \mu)$ is represented by an element in the admissible set. 
\end{itemize}

We discuss the first ingredient in this subsection and the second and third ingredients in the next subsection. 

The starting point is the natural map $$\Psi: B(\tW, \s) \to B(G)$$ from the set of $\s$-conjugacy classes of $\tW$ to the set of $\s$-conjugacy classes of $G(L)$. This map is surjective, but not injective in general. However, there exists a natural section of $\Psi$ given by the straight $\s$-conjugacy classes of $\tW$ (see \S\ref{tri}). 

On the set of straight $\s$-conjugacy classes of $\tW$, there is a natural partial order $\preceq_\s$ (defined in \S\ref{5.4}). On $B(G)$, there are two partial orders, given by the closure relation between the $\s$-conjugacy classes and given by the dominance order of the corresponding Newton polygons. A generalization of Grothendieck conjecture says that the two partial orders on $B(G)$ coincide. We prove in Theorem \ref{tri-order} that 

\begin{thrm2}
For any twisted loop group, the partial order $\preceq_\s$ on the set of straight $\s$-conjugacy classes coincides with both partial orders on $B(G)$ via the map $\Psi: B(\tW, \s) \to B(G)$. In particular, the two partial orders on $B(G)$ coincide. 
\end{thrm2}

The proof is based on the reductive method in \cite{He99} \`a la Deligne and Lusztig, some remarkable combinatorial properties on $\tW$ established in \cite{HN} and the Grothendieck conjecture for split groups proved by Viehmann in \cite{Vi1}. 

%In fact, the original Grothendieck conjecture and the theorem in \cite{Vi1} is about the closure relation between the Newton strata with hyperspecial level structure, not the closure relation between the $\s$-conjugacy classes. However, the $\mu$-ordinary locus (i.e., the maximal Newton strata with respect to the dominance order) is not dense in the reduction of Shimura varieties with other parahoric level structures for the quasi-split unramified groups, I don't know the right formulation for the generalization of Grothendieck's conjecture on Newton strata for general twisted loop groups.\footnote{However, one may ask when the closures of two Newton strata intersect. This question will be studied elsewhere.}

\subsection{} By definition, $$X(\mu, b)_J \neq \emptyset \Leftrightarrow [b] \cap \cup_{w \in \Adm_J(\mu)} P_J w P_J \neq \emptyset.$$ Using a similar argument as in the proof of Theorem B, the latter condition is equivalent to $[b] \in \Psi(\Adm_J(\mu))$. 

Notice that Mazur's inequality is defined using the dominance order on the Newton polygons. For quasi-split groups, it is easy to see that $\mu$ is the unique maximal element in $B(G, \mu)$ with respect to the dominance order. Thus the converse to Mazur's inequality follows from the coincides between the partial order $\preceq_\s$ on the set of straight $\s$-conjugacy classes and the dominance order on the Newton polygons. For non quasi-split groups, the maximal element in $B(G, \mu)$ is harder to understand and we use \cite{HN-adm} on the properties of this element. 

The proof of Mazur's inequality is based on two properties of the admissible sets:

\begin{itemize}

\item The additivity of the admissible sets (Theorem \ref{add}), proved by Zhu's global Schubert varieties \cite{Z}.

\item The compatibility of admissible sets (Theorem \ref{comp}), proved by the ``partial conjugation method'' in \cite{HeMin}.

\end{itemize}

\section{Preliminaries}

\subsection{} Let $\mathbb F_q$ be the finite field with $q$ elements. Let $\kk$ be an algebraic closure of $\mathbb F_q$. Let $F$ be a finite field extension of $\mathbb Q_p$ with residue class field $\mathbb F_q$ and uniformizer $\varepsilon$ or $F= \mathbb F_q( (\e))$ be the field of Laurent series over $\mathbb F_q$. Let $L$ be the completion of the maximal unramified extension of $F$.

Let $G$ be a connected semisimple group over $F$ which splits over a tamely ramified extension of $F$.  Let $\s$ be the Frobenius automorphism of $L/F$. We also denote the induced automorphism on $G(L)$ by $\s$.

Let $S$ be a maximal $L$-split torus that is defined over $F$ and let $T$ be its centralizer. By Steinberg's theorem, $G$ is quasi-split over $L$. Thus $T$ is a maximal torus. Let $N$ be its normalizer. The {\it finite Weyl group} associated to $S$ is $$W_0=N(L)/T(L).$$ The Iwahori-Weyl group associated to $S$ is $$\tW=N(L)/T(L)_1,$$ where $T(L)_1$ denotes the unique Iwahori subgroup of $T(L)$. The Frobenius morphism $\s$ induces an action on $\tW$, which we still denote by $\s$. 

For any $w \in \tW$, we choose a representative in $N(L)$ and also write it as $w$. 

\subsection{} Let $\ca$ be the apartment of $G_L$ corresponding to $S$. Since $\s$ induces a permutation of finite order on the set of alcoves in $\ca$, there exists a $\s$-invariant alcove $\aa$ in $\ca$. Let $I$ be the corresponding Iwahori subgroup. Let $\tSS$ be the set of simple reflections of $\tW$. The set $\tSS$ is equipped with an action of $\s$. For any $J \subset \tSS$, let $W_J \subset \tW$ be the subgroup generated by the simple reflections in $J$ and by ${}^J \tW$ (resp. $\tW^J$) the set of minimal length elements for the cosets $W_J \backslash \tW$ (resp. $\tW/W_J$). We simply write ${}^J \tW^{J'}$ for ${}^J \tW \cap \tW^{J'}$. 

We follow \cite{HR}. Let $\G_F=\Gal(\bar L/F)$ be the absolute Galois group of $F$ and $\G=\Gal(\bar L/L)$ the inertia group. The Iwahori-Weyl group $\tW$ contains the affine Weyl group $W_a$ as a normal subgroup and we have a short exact sequence $$0 \to W_a \to \tW \to \pi_1(G)_\G \to 0,$$ where $\pi_1(G)$ denotes algebraic fundamental group of $G$ and $\pi_1(G)_\G$ its coinvariants under the action of $\s$. The choice of the alcove $\aa$ splits this extension, and $$\tW=W_a \rtimes \Omega,$$ where $\Omega$ is the normalizer of $\aa$, and is isomorphic to $\pi_1(G)_\G$. The length function and Bruhat order on $W_a$ extend in a natural way to $\tW$. 

We have another exact sequence $$0 \to X_*(T)_\G \to \tW \to W_0 \to 0.$$ We choose a special vertex of $\aa$ and represent $\tW$ as a semidirect product $$\tW=X_*(T)_\G \rtimes W_0=\{t^\l w; \l \in X_*(T)_\G, w \in W_0\}.$$

\subsection{} For $b, b' \in G(L)$, we say that $b$ and $b'$ are $\s$-conjugate if there exists $g \in G(L)$ such that $b'=g \i b \s(g)$. Let $B(G)$ be the set of $\s$-conjugacy classes. The classification of the $\s$-conjugacy classes is obtained by Kottwitz in \cite{Ko1} and \cite{Ko2}. The description is as follows. 

Let $\k_G: B(G) \to \pi_1(G)_{\G_F}$ be the Kottwitz map \cite[\S 7]{Ko2}. This gives one invariant. Another invariant is obtained by the Newton map. An element $b \in G(L)$ determines a homomorphism $\mathbb D \to G_L$, where $\mathbb D$ is the pro-algebraic torus whose character group is $\QQ$. This homomorphism determines an element $\nu_b$ in the closed dominant chamber $X_*(T)_\QQ^+$. The element $\nu_b$ is called the {\it Newton point} of $b$ and the map $b \mapsto \nu_b$ is called the {\it Newton map}. Note that for any $b$, $\s(\nu_b)=\nu_b$. By \cite[\S 4.13]{Ko2}, the map $$f: B(G) \to X_*(T)_\QQ^+ \times \pi_1(G)_{\G_F}, \qquad b \mapsto (\nu_b, \k_G(b))$$ is injective. 

\subsection{} Write $\s$ as $\s=\t \circ \s_0$, where $\s_0$ is a diagram automorphism of $G(L)$ such that $\s_0$ fixes $\tSS-\SS$ and the induced action of $\t$ on the adjoint group $G_{ad}$ is inner. 

For $\nu, \nu' \in X_*(T)_\QQ^+$, we write $\nu \le \nu'$ if $\nu'-\nu$ is a non-negative $\QQ$-linear combination of positive relative coroots. This is called the dominance order on $X_*(T)_\QQ^+$. 

Let $\mu \in X_*(T)^+$, we set $$\mu^\diamond=\frac{1}{N} \sum_{i=0}^{N-1} \s_0^i(\mu) \in X_*(T)_\QQ^+,$$ where $N$ is the order of $\s_0$. A $\s$-conjugacy class $[b]$ is called {\it (neutral) acceptable} for $\mu$ if $\nu_b \le \mu^\diamond$ and $\k_G(b)=\mu^\sharp$, where $\mu^\sharp$ is the image of $\mu$ in $\pi_1(G)_{\G_F}$. Let $B(G, \mu)$ be the set of (neutral) acceptable elements for $\mu$. 

\subsection{}\label{1.5} The {\it $\mu$-admissible set} is defined as $$\Adm(\mu)=\{w \in \tW; w \le t^{x(\underline \mu)} \text{ for some }x \in W_0\},$$ where $\underline \mu$ is the image of $\mu$ in $X_*(T)_\G$. 

More generally, let $J \subset \tSS$ such that $\s(J)=J$ and $W_J$ is finite. The $\mu$-admissible set associated to $J$ is $$\Adm^J(\mu)=W_J \Adm(\mu) W_J \subset \tW.$$ It is the inverse image under the natural map $\tW \to W_J \backslash \tW/W_J$ of $\Adm_J(\mu)$ in \cite[(3.6)]{KR}. 

\subsection{} Let $J \subset \tSS$ such that $\s(J)=J$ and $W_J$ is finite. Let $P_J \supset I$ be the standard parahoric subgroup corresponding to $J$. For any $w \in W_J \backslash \tW/W_J$ and $b \in G(L)$, the generalized affine Deligne-Lusztig variety $$X_{J, w}(b)=\{g \in G(L)/P_J; g \i b \s(g) \in P_J w P_J\}.$$ In this paper, we are mainly interested in the following finite union of affine Deligne-Lusztig varieties: \begin{align*} X(\mu, b)_J &=\{g \in G(L)/P_J; g \i b \s(g) \in Y_{J, \mu}\} \\ &=\cup_{w \in \Adm^J(\mu)} X_{J, w}(b),\end{align*} where $Y_{J, \mu}=\cup_{w \in \Adm(\mu)} P_J w P_J=\cup_{w \in \Adm^J(\mu)} I w I.$ 

Let $J' \subset \tSS$ such that $\s(J')=J'$ and $W_{J'}$ is finite and $J \subset J'$. Then $Y_{J', \mu}=P_{J'} Y_{J, \mu} P_{J'}$ and hence the projection map $G(L)/P_J \to G(L)/P_{J'}$ induces $$\pi_{J, J'}: X(\mu, b)_J \to X(\mu, b)_{J'}.$$

The main result of this paper is 

\begin{thm}\label{main}
Let $b \in G(L)$, $\mu \in X_*(T)^+$ and $J \subset J'$ be $\s$-stable subsets of $\tSS$ with $W_{J'}$ finite. Then 

(1) $X(\mu, b)_J \neq \emptyset$ if and only if $[b] \in B(G, \mu)$. 

(2) The map $\pi_{J, J'}$ is surjective. 
\end{thm}

\section{The map $\Psi: B(\tW, \s) \to B(G)$}

\subsection{} We first recall the definition of straight elements of $\tW$. 

Let $w \in \tW$. Then there exists a positive integer $n$ such that $(w \s)^n=t^\l \in \tW \rtimes \<\s\>$ for some $\l \in X_*(T)_\G$. Let $\nu_{w, \s}=\l/n$ and $\bar \nu_{w, \s}$ be the unique dominant element in the $W_0$-orbit of $\nu_{w, \s}$. It is known that $\bar \nu_{w, \s}$ is independent of the choice of $n$ and is $\G$-invariant. 

We say that an element $w$ is {\it $\s$-straight} $\ell(w)=\<\bar \nu_{w, \s}, 2 \rho\>$, where $\rho$ is the half sum of all positive roots in the root system of the affine Weyl group $W_a$. This is equivalent to $\ell((w\s)^n)=n \ell(w)$, where we regard $w \s$ as an element in $\tW \rtimes \<\s\>$. A $\s$-conjugacy class of $\tW$ is called {\it straight} if it contains a $\s$-straight element. 

\subsection{}\label{tri} Let $B(\tW, \s)$ be the set of $\s$-conjugacy classes of $\tW$ and $B(\tW, \s)_{str}$ be the set of straight $\s$-conjugacy classes of $\tW$. Following \cite{He99}, there exists a commutative diagram
\[\tag{a} \xymatrix{B(\tW, \s) \ar[rr]^{\Psi} \ar[dr]_f & & B(G) \ar[ld]^-f \\ & X_*(T)_\QQ^+ \times \pi_1(G)_{\G_F} &,}\] where $\Psi: B(\tW, \s) \to B(G)$ is induced from the natural inclusion $N(L) \to G(L)$. 

By \cite[\S 3]{He99}, the restriction of $\Psi$ to $B(\tW, \s)_{str}$ is a bijection. For any straight $\s$-conjugacy class $\co$ of $\tW$, we denote by $[\co]$ the corresponding $\s$-conjugacy class in $G(L)$. We also set $\nu_\co=\bar \nu_{w, \s}$ for any $w \in \co$. 

\subsection{} By definition, for $w \in \tW$, $X_{\emptyset, w}(b) \neq \emptyset$ if and only if $[b] \in I w I \neq \emptyset$. If $\Psi(w)=[b]$, then automatically $[b] \cap I w I \neq \emptyset$, i.e. $X_{\emptyset, w}(b) \neq \emptyset$. The converse, is far from being true. In \cite[Theorem 6.1]{He99}, we give a criterion about the non-emptiness pattern of affine Deligne-Lusztig varieties in affine flag varieties in terms of class polynomials of affine Hecke algebras. The computation of class polynomials, however, is very hard in general.

The main result of this section is the following simple criterion of the non-emptiness criterion for ``closed'' affine Deligne-Lusztig varieties in affine flag varieties. 

\begin{thm}\label{w-closure}
Let $b \in G(L)$ and $w \in \tW$. Then $\cup_{w' \le w} X_{\emptyset, w'}(b) \neq \emptyset$ if and only if $[b] \in \cup_{w' \le w} \Psi(w')$. 

%Let $w \in \tW$. A $\s$-conjugacy class of $G(L)$ that intersects $\cup_{w' \le w} I w I$ if and only if it is represented by a straight $\s$-conjugacy class $\co$ of $\tW$ with $\co \preceq_\s w$. 
\end{thm}

To prove this theorem, we combine the method for the finite case \cite[Proposition 5.8]{HeMin} and \cite[Proposition 2.5]{He-aff}, with the reduction method \cite[Section 3]{He99}. The proof will be given in \S \ref{proof-w-closure}. 

\subsection{} For $w, w' \in \tW$ and $s \in \tSS$, we write $w \xrightarrow{s}_\s w'$ if $w'=s w \s(s)$ and $\ell(w') \le \ell(w)$. We write $w \to_\s w'$ if there is a sequence $w=w_0, w_1, \cdots, w_n=w'$ of elements in $\tW$ such that for any $k$, $w_{k-1} \xrightarrow{s}_\s w_k$ for some $s \in \tSS$. We write $w \approx_\s w'$ if $w \to_\s w'$ and $w' \to_\s w$ and write $w \tilde \approx_\s w'$ if $w \approx_\s \t w' \s(\t) \i$ for some $\t \in \Omega$. It is easy to see that $w \approx_\s w'$ if $w \to_\s w'$ and $\ell(w)=\ell(w')$.

For any $\s$-conjugacy class $\co$ in $\tW$, we denote by $\co_{\min}$ the set of minimal length elements in $\co$. Now we recall some properties on the minimal length elements, obtained in \cite[\S 2]{HN}. 

\begin{thm}\label{ux}
Let $\co$ be a $\s$-conjugacy class of $\tW$ and $w \in \co$. Then there exists $w' \in \co_{\min}$ such that 

(1) $w \to_\s w'$;

(2) There exists $J \subset \tSS$ with $W_J$ finite, an $\s$-straight element $x \in \tW$ with $x \in {}^J \tW^{\s(J)}$ and $x \s(J)=J$, and $u \in W_J$, such that $w'=u x$. 
\end{thm}

\begin{thm}\label{min-str}
Let $\co$ be a straight $\s$-conjugacy class of $\tW$ and $w, w' \in \co_{\min}$. Then $w \tilde \approx_\s w'$. 
\end{thm}

\subsection{}\label{g-stable} For $g, g' \in G(L)$, we write $g \cdot_\s g'=g g' \s(g) \i$. The subset $G(L) \cdot_\s I w I$ is studied in \cite[\S 3]{He99}. Now we recollect some results that will be used here. 

(1) If $w \tilde \approx_\s w'$, then $G(L) \cdot_\s I w I=G(L) \cdot_\s I w' I$. 

(2) If $w \in W$ and $s \in \tSS$ such that $s w \s(s)<w$, then $G(L) \cdot_\s I w I=G(L) \cdot_\s I s w I \cup G(L) \cdot_\s I s w \s(s) I$. 

(3) If $w \in W$ is a minimal length element in its $\s$-conjugacy class, then $G(L) \cdot_\s I w I$ is a single $\s$-conjugacy class in $G(L)$. 

(4) Let $J \subset \tSS$ with $W_J$ finite, and $x \in \tW$ with $x \in {}^J \tW^{\s(J)}$ and $x \s(J)=J$. Then for any $u \in W_J$, $G(L) \cdot_\s I u x I=G(L) \cdot_\s I x I$.

\subsection{} Let $w \in \tW$ and $\co$ be a straight $\s$-conjugacy class in $\tW$. We write $\co \preceq_\s w$ if there exists a minimal length element $w' \in \co$ such that $w' \le w$ in the usual Bruhat order. 

Now we discuss some properties on $\preceq_\s$. 

\begin{prop}\label{preceq}
Let $w, w' \in \tW$ with $w \to_\s w'$. Let $\co$ be a straight $\s$-conjugacy class of $\tW$. If $\co \preceq_\s w'$, then $\co \preceq_\s w$. 
\end{prop}

\begin{rem}
The proof is similar to the finite case \cite[Lemma 4.4]{HeMin}. We include the proof here for completeness. 
\end{rem}

\begin{proof}
It suffices to prove the case where $w'=s w \s(s)$ for some $s \in \tSS$. Let $x \in \co_{\min}$ with $x \le w'$. 

If $w>w'$, then $x<w$ and hence $\co \preceq_\s w$. Now we assume that $\ell(w)=\ell(w')$. Without loss of generalization, we may assume that $s w<w$ and $w \s(s)>w$. 

If $s x<x$, then $\ell(s x \s(s)) \le \ell(x)$. Since $x \in \co_{\min}$, $s x \s(s) \in \co_{\min}$. By \cite[Corollary 2.5]{Lu-hecke}, $s x \le s w$ and $s x \s(s) \le s w \s(s)$. Hence $\co \preceq_\s w$. 

If $s x>x$, then \cite[Corollary 2.5]{Lu-hecke}, $x \le s w$ and hence $x<s w \s(s)$. We also have that $\co \preceq_\s w$. 
\end{proof}

\begin{cor}\label{min-min}
Let $\co$ be a straight $\s$-conjugacy class of $\tW$ and $w \in \co$. Then $v$ is of minimal length in $\co$ if and only if $v$ is a minimal element in $\co$ with respect to the Bruhat order. 
\end{cor}

\begin{cor}\label{preceq'}
Let $\co$ be a straight $\s$-conjugacy class of $\tW$ and $w \in \tW$. Then $\co \preceq_\s w$ if and only if there exists $x \in \co$ with $x \le w$. 
\end{cor}

\subsection{}\label{proof-w-closure} Now we prove Theorem \ref{w-closure}. 

By definition, $\cup_{\co \preceq_\s w} [\co] \subset \cup_{w' \le w} \Psi(w') \subset \cup_{w' \le w} G(L) \cdot I w' I$. 

Now we show that $\cup_{w' \le w} G(L) \cdot I w' I \subset \cup_{\co \preceq_\s w} [\co]$. By induction, it suffices to show that $$G(L) \cdot I w I \subset \cup_{\co \preceq_\s w} [\co].$$ We argue by induction on $\ell(w)$. 

If $w$ is of minimal length in its $\s$-conjugacy class, then by Theorem \ref{ux} (2), then there exists $J \subset \tSS$ with $W_J$ finite, $x \in \tW$ be an $\s$-straight element with $x \in {}^J \tW^{\s(J)}$ and $x \s(J)=J$, and $u \in W_J$ such that $w \approx_\s ux$. Let $\co_x$ be the $\s$-conjugacy class of $x$. Then $\co_x \preceq_\s u x$. By Proposition \ref{preceq}, $\co_x \preceq_\s w$. By \S \ref{g-stable} (1), (3) \& (4), $$G(L) \cdot_\s I w I=G(L) \cdot_\s I u x I=G(L) \cdot_\s I x I=[\co_x] \subset \cup_{\co \preceq_\s w} [\co].$$

If $w$ is not of minimal length in its $\s$-conjugacy class, then by Theorem \ref{ux} (1), there exists $w' \in \tW$ with $w \approx_\s w'$ and $s \in \tSS$ with $s w' \s(s)<w'$. By \S \ref{g-stable} (1) \& (2), $$G(L) \cdot_\s I w I=G(L) \cdot_\s I w' I=G(L) \cdot_\s I s w' I \cup G(L) \cdot_\s I s w' \s(s) I.$$

By induction hypothesis on $s w'$ and $s w' \s(s)$, $$G(L) \cdot_\s I w I \subset \cup_{\co \preceq_\s s w' \text{ or } \co \preceq_\s s w' \s(s)} [\co] \subset \cup_{\co \preceq_\s w'} [\co].$$ 

By Proposition \ref{preceq}, $\co \preceq_\s w'$ if and only if $\co \preceq_\s w$. Hence $G(L) \cdot_\s I w I \subset \cup_{\co \preceq_\s w} [\co]$. The statement is proved. 

\begin{cor}\label{2.8}
Let $w \in \tW$. Then $\cup_{w' \le w} \Psi(w')=\cup_{\co \preceq_\s w} [\co]$. 
\end{cor}

The following special case of Theorem \ref{w-closure} is useful in this paper. 

\begin{cor}\label{mu-closure}
Let $b \in G(L)$, $\mu \in X_*(T)^+$ and $J \subset \tSS$ such that $\s(J)=J$ and $W_J$ is finite. Then $X(\mu, b)_J \neq \emptyset$ if and only if $[b] \in \Psi(\Adm^J(\mu))$.
\end{cor}

\begin{proof}
By definition, $X(\mu, b)_J \neq \emptyset$ if and only if $$[b] \subset \cup_{w \in \Adm^J(\mu)} G(L) \cdot_\s I w I.$$

Notice that $\Adm^J(\mu)$ is of the form $\cup_i \{w \in \tW; w \le x_i\}$ for finitely many $x_i$'s. The statement follows from Theorem \ref{w-closure}.
\end{proof}

\section{Three partial orders}

\subsection{} In this section, we assume that $F=\FF_q((\e))$. Recall the commutative diagram in \S\ref{tri} (a):
\[\xymatrix{B(\tW, \s)_{str} \ar[rr]^{\Psi} \ar[dr]_f & & B(G) \ar[ld]^-f \\ & X_*(T)_\QQ^+ \times \pi_1(G)_{\G_F} &,}\]

We will introduce partial orders on these sets and show that these partial orders are compatible. 

\subsection{}\label{5.4} Let $\co, \co' \in B(\tW, \s)_{str}$. We write $\co' \preceq_\s \co$ if for some $w \in \co_{\min}$, $\co' \preceq_\s w$. By Theorem \ref{min-str} and Proposition \ref{preceq}, if $\co' \preceq_\s \co$, then $\co' \preceq_\s x$ for any $x \in \co_{\min}$. Hence $\preceq_\s$ is a partial order on $B(\tW, \s)_{str}$. 

For $(v_1, \k_1), (v_2, \k_2) \in X_*(T)_\QQ^+ \times \pi_1(G)_{\G_F}$, we write $(v_1, k_1) \le (v_2, k_2)$ if $v_1 \le v_2$ (the dominance order) and $k_1=k_2$. 

Following Grothendieck, we introduce admissible subscheme of $G(L)$ and show that each $\s$-conjugacy class of $G(L)$ is a locally closed admissible subscheme of $G(L)$ (see Appendix). The closure relation between the $\s$-conjugacy classes of $G(L)$ gives a partial order on $B(G)$. 

The main result of this section is

\begin{thm}\label{tri-order} Let $\co, \co' \in B(\tW, \s)_{str}$. The following conditions are equivalent:

(1) $\co \preceq_\s \co'$. 

(2) $[\co] \subset \overline{[\co']}$. 

(3) $f(\co) \le f(\co')$, i.e. $\k(\co)=\k(\co')$ and $\nu_\co \le \nu_{\co'}$. 
\end{thm}

\begin{proof}
We first prove $(1) \Leftrightarrow (2)$. 

Let $w'$ be a $\s$-straight element of $\co'$. Then 
\begin{align*}
\overline{[\co']} &=\overline{G(L) \cdot_\s I w' I}=\cup_{w \le w'} G(L) \cdot_\s I w I \\ &=\cup_{\co_1 \preceq_\s w'} [\co_1]=\cup_{\co_1 \preceq \co'} [\co_1].
\end{align*}
Here the first equality follows from \S\ref{g-stable} (3), the second equality follows from Theorem \ref{w-closure'}, the third equality follows from Theorem \ref{w-closure} and Corollary \ref{2.8} and the last equality follows from \S\ref{5.4}. 

Next we prove $(1) \Rightarrow (3)$. 

If $\co \preceq_\s \co'$, then there exists $w \in \co_{\min}$ and $w' \in \co'_{\min}$ such that $w \le w'$. In particular, $w W_a=w' W_a$. Hence $\k(\co) =\k(\co')$. Moreover, $w$ and $w'$ are $\s$-straight elements. So for any $n$, $\ell((w \s)^n)=n \ell(w)$ and $\ell((w' \s)^n)=n \ell(w')$. Thus $(w \s)^n \le (w' \s)^n$. In particular, $t^{m \nu_{w, \s}} \le t^{m \nu_{w', \s}}$ for sufficiently divisible integer $m$. In particular, $m \bar \nu_{w, \s} \le m \bar \nu_{w', \s}$. So $\bar \nu_{w, \s} \le \bar \nu_{w', \s}$.

Now we prove $(3) \Rightarrow (1)$. 

Suppose that $\k(\co) =\k(\co')$ and $\nu_{\co} \le \nu_{\co'}$. Let $\tW_{ad}$ be the Iwahori-Weyl group of the adjoint group $G_{ad}$. The natural projection $\pi: \tW \to \tW_{ad}$ send $\co$ to $\co_{ad}$ and $\co'$ to $\co'_{ad}$. As $\pi$ preserves length, $\co_{ad}$ and $\co'_{ad}$ are straight $\s$-conjugacy classes of $\tW_{ad}$. Moreover, $\k(\co_{ad}) =\k(\co'_{ad})$ and $\nu_{\co_{ad}} \le \nu_{\co'_{ad}}$.

We may write $\s$ as $\s=\text{Ad} (\t) \circ \s_0$, where $\t$ is a length-zero element in $\tW_{ad}$ and $\s_0$ is a diagram automorphism of $\tW_{ad}$ such that $\s_0$ fixes $\tSS-\SS$. Then $\co_{ad} \t$ and $\co'_{ad} \t$ are straight $\s_0$-conjugacy classes of $\tW_{ad}$. Moreover, $\nu_{\co_{ad} \t} \le \nu_{\co'_{ad} \t}$.

We associate a quasi-split unramified group $H$ to the pair $(\tW_{ad}, \s_0)$. We regard $[\co_{ad} \t]$ and $[\co'_{ad} \t]$ as $\s_0$-conjugacy classes of $H(L)$. By \cite[Theorem 2]{Vi1} and \cite[Theorem 1.1]{Ha}\footnote{The statement in \cite{Ha} is for PEL type Shimura varieties. The argument still holds for any unramified loop groups over function fields. It is based on Viehmann's strategy in \cite[Proof of Theorem 20]{Vi1} (see also \cite[Proposition 5.13]{Ha}, using the dimension formula of affine Deligne-Lusztig varieties \cite{Ha1} and the purity Theorem \cite[Corollary 18]{Vi1} and \cite[Proposition 5.4]{Ha}.}, $[\co_{ad} \t] \subset \overline{[\co'_{ad} \t]}$. By the equivalence $(1) \Leftrightarrow (2)$ for $G_{ad}$, $\co_{ad} \t \preceq_{\s_0} \co'_{ad} \t$. This is equivalence to $\co_{ad} \preceq_\s \co'_{ad}$. 

By definition, there exists $w_{ad} \in (\co_{ad})_{\min}$ and $w'_{ad} \in (\co'_{ad})_{\min}$ such that $w_{ad} \le w'_{ad}$. Let $w \in \co$ and $w' \in \co'$ such that $\pi(w)=w_{ad}$, $\pi(w')=w'_{ad}$ and $w W_a=w' W_a$. Then $w \le w'$. Hence $\co \preceq_\s \co'$. 
\end{proof}

\section{Converse to Mazur's inequality}

\begin{prop}\label{5.1}
Let $\mu \in X_*(T)^+$ and $\co$ be a straight $\s$-conjugacy class of $\tW$. If $\k(\co)=\mu^\sharp$ and $\nu_\co \le \mu^\diamond$, then $\co \cap \Adm(\mu) \neq \emptyset$. 
\end{prop}

\begin{proof}
By \cite{HN-adm}, the set $\{\nu_\co; \k(\co)=\mu^\sharp, \nu_\co \le \mu^\diamond\}$ contains a unique maximal element $\nu$ and there exists $x \in \Adm(\mu)$ with $\bar \nu_x=\nu$. 

Let $\co$ be a straight $\s$-conjugacy class $\co$ of $\tW$ with $\k(\co)=\mu^\sharp$ and $\nu_\co \le \mu^\diamond$. Then $\nu_\co \le \nu$. By Theorem \ref{tri-order} and Corollary \ref{min-min}, $\nu_\co \le \mu^\diamond$, $\co \preceq_\s x$. In other words, there exists $w \in \co_{\min}$ such that $w \le x$. Since $\Adm(\mu)$ is closed under the Bruhat order, $w \in \Adm(\mu)$. 
\end{proof}

\

Now we prove the converse to Mazur's inequality. 

\begin{thm}\label{converse-Mazur}
Let $b \in G(L)$, $\mu \in X_*(T)^+$ and $J \subset \tSS$ such that $\s(J)=J$ and $W_J$ is finite. If $b \in B(G, \mu)$, then $X(\mu, b)_J \neq \emptyset$. 
\end{thm}

\begin{proof}
Let $b \in B(G, \mu)$. Then $[b]$ is represented by a straight $\s$-conjugacy class $\co$ of $\tW$. By Proposition \ref{5.1}, $\co \cap \Adm(\mu) \neq \emptyset$. Note that $\Adm(\mu) \subset \Adm^J(\mu)$. Hence $\co \cap \Adm^J(\mu) \neq \emptyset$. By Corollary \ref{mu-closure}, $X(\mu, b)_J \neq \emptyset$. 
\end{proof}

\section{Mazur's inequality: Iwahori case} 

To prove Mazur's inequality in the Iwahori case, we need the following additivity property of admissible sets due to Xinwen Zhu \cite{Z2}. 

\begin{thm}\label{add}
Let $\mu, \mu' \in X_*(T)^+$. Then $$\Adm(\mu) \Adm(\mu')=\Adm(\mu+\mu').$$ 
\end{thm}

\begin{proof}
We first show that $\Adm(\mu+\mu') \subset \Adm(\mu) \Adm(\mu')$. 

Let $z \in \Adm(\mu+\mu')$. By definition, $z \le t^{x(\underline \mu+\underline \mu')}$ for some $x \in W_0$. Notice that $t^{x(\underline \mu+\underline \mu')}=t^{x(\underline \mu)} t^{x(\underline \mu')}$ and $\ell(t^{x(\underline \mu+\underline \mu')})=\ell(t^{x(\underline \mu)}) \ell(t^{x(\underline \mu')})$. In other words, there exists a reduced expression of $t^{x(\underline \mu+\underline \mu')}$ consisting of two parts, the first part is a reduced expression of $t^{x(\underline \mu)}$ and the second part is a reduced expression of  $t^{x(\underline \mu')}$. Hence there exists $z_1 \le t^{x(\underline \mu)} \in \Adm(\mu)$ and $z_2 \le t^{x(\underline \mu')} \in \Adm(\mu')$ such that $z=z_1 z_2$. 

The proof of the other direction $\Adm(\mu) \Adm(\mu') \subset \Adm(\mu+\mu')$ is based on the theory of global Schubert varieties of Zhu \cite{Z}. We first recall the definition. 

Let $L=\bar \FF_q((\e))$ and $G$ be a connected reductive group over $L$, split over a tamely ramified extension, and with Iwahori-Weyl group $\tW$. Let $\cg$ be the Iwahori group scheme over $\co_L$. The element $\mu \in X_*(T)$ defines a section $s_\mu$ of the global affine Grassmannian $Gr_{\cg}$ as in \cite[Proposition 3.4]{Z}. The global Schubert variety $\overline{Gr_{\cg, \mu}}$ is the scheme-theoretic closure of the $\cl^+ \cg \cdot s_\mu$ in $Gr_\cg$, where $\cl^+\cg$ is the positive loop group. It is a scheme over $\co_F$. One of the main result of \cite{Z} is that the special fiber of $\overline{Gr_{\cg, \mu}}$  is isomorphic to $\sqcup_{w \in \Adm(\mu)} I w I/I$. 

Now we take the convolution product of $\overline{Gr_{\cg, \mu}}$ with $\overline{Gr_{\cg, \mu}}$ as in \cite[\S 6]{Z}. By definition, the special fiber of the convolution product is isomorphic to $\cup_{w \in \Adm(\mu), w' \in \Adm(\mu')} I w I w' I/I \supset \cup_{z \in \Adm(\mu) \Adm(\mu')} I z I/I$. On the other hand, it is proved in \cite[\S 6]{Z} that the special fiber is isomorphic to $\sqcup_{z \in \Adm(\mu+\mu')} I z I$. Hence $\Adm(\mu) \Adm(\mu') \subset \Adm(\mu+\mu')$. 
%Let $L=\FF_q((\e))$ and $H$ be a split group over $L$ with Iwahori-Weyl group $\tW$. Set $K=H(\FF_q[[\e]])$. Then $H(L)=\sqcup_{\l \in X_*(T)^+} K \e^\l K$. Moreover, $$\overline{K \e^\mu K \e^{\mu'} K/K}=\overline{K \e^{\mu+\mu'} K/K}=\sqcup_{\mu'' \in X_*(T)^+, \mu'' \le \mu+\mu'} K \e^{\mu''} K/K$$ in the affine Grassmannian $H(L)/K$. Since $w \in K \e^\l K$ and $w' \in K \e^{\l'} K$, $w w' \in K \e^{\mu''} K/K$ for some $\mu'' \in X_*(T)^+$ with $\mu'' \le \mu+\mu'$. Thus $w w' \in W_0 t^{\mu''} W_0$. 
\end{proof}

Now we prove Mazur's inequality in the Iwahori case. 

\begin{thm}\label{Mazur}
Let $b \in G(L)$ and $\mu \in X_*(T)^+$. If $X(\mu, b)_\emptyset \neq \emptyset$, then $b \in B(G, \mu)$. 
\end{thm}

%\begin{thm}\label{nu-le-mu}
%Let $\mu \in X_*(T)^+$ and $\co$ be a straight $\s$-conjugacy class. If $\co \cap \Adm(\mu) \neq \emptyset$, then $\nu_\co \le \mu^\diamond$. 
%\end{thm}

\begin{proof}
Recall that $\s=\t \circ \s_0$, where $\s_0$ is a diagram automorphism of $G(L)$ such that $\s_0$ fixes $\tSS-\SS$ and the induced action of $\t$ on the adjoint group $G_{ad}$ is inner. For any $\mu \in X_*(T)^+$, $\s_0(\Adm(\mu))=\Adm(\s_0(\mu))$. Note that $\t(\mu)=x(\mu)$ for some $x \in W_0$. Thus $\t(\Adm(\mu))=\Adm(\mu)$.  Therefore $$\s(\Adm(\mu))=\Adm(\s_0(\mu)).$$ 

By Theorem \ref{w-closure}, $X(\mu, b)_\emptyset \neq \emptyset$ implies that $w \in [b]$ for some $w \in \Adm(\mu)$. Let $n_0$ be the order of $\s$ in $\text{Aut}(\tW)$ and $n=n_0 \sharp(W_0)$. We regard $w \s$ as an element in $\tW \rtimes \<\s\>$. Then $(w \s)^{n_0} \in \tW$ and $(w \s)^n=t^\l$ for some $\l \in X_*(T)$. By definition, $\l$ lies in the $W_0$-orbit of $n \nu_\co$. On the other hand, 

\begin{align*}
(w \s)^n &=w \s(w) \s^2(w) \cdots \s^{n-1}(w) \\ & \in \Adm(\mu) \Adm(\s_0(\mu)) \Adm(\s_0^2(\mu)) \cdots \Adm(\s_0^{n-1}(\mu)) \\ &=\Adm(\mu+\s_0(\mu)+\cdots+\s_0^{n-1}(\mu)) \\ &=\Adm(n \mu^\diamond)
\end{align*}

Hence $t^{\l} \in \Adm(n \mu^\diamond)$ and $\bar \l \le n \mu^\diamond$. Thus $\nu_\co \le \mu^\diamond$. 
\end{proof}

\section{Mazur's inequality: General case}

\subsection{}\label{6-1} To pass from Iwahori case to the general case, we need part (2) of Theorem \ref{main}. There are two key ingredients in the proof. 

\begin{enumerate}[(a)]

\item A suitable stratification of $Y_{J, \mu}$ with respect to the $\s$-conjugation action of $P_J$. 

\item A compatibility property of admissible sets.

\end{enumerate}

\subsection{}\label{6-3} We discuss \S\ref{6-1}(a) first. The stratification is established in \cite[\S2]{He-aff} and \cite[\S3]{GH}, generalizing Lusztig's $G$-stable piece decomposition for the finite case. 

Let $J=\s(J) \subset \tSS$ with $W_J$ finite. For any $w \in {}^J \tW$, we consider the subset $P_J \cdot_\s I w I$ of $G(L)$. Then

\begin{enumerate}

\item $G(L)=\sqcup_{w \in {}^J \tW} P_J \cdot_\s I w I$. 

\item $Y_{J, \mu}=\sqcup_{w \in {}^J \tW \cap \Adm^J(\mu)} P_J \cdot_\s I w I$. 

\item If $F=\FF_q((\e))$, then for $w \in {}^J \tW$, $$\overline{P_J \cdot_\s I w I}=\sqcup_{w'} P_J \cdot_\s I w' I,$$ where $w'$ runs over elements in ${}^J \tW$ such that there exists $x \in W_J$ with $x w \s(x) \i \le w'$. 

\end{enumerate}

Then we discuss the following compatibility result on the sets ${}^J \tW \cap \Adm^J(\mu)$. 

\begin{thm}\label{comp}
Let $\mu \in X_*(T)^+$ and $J \subset \tSS$ with $W_J$ finite. Then ${}^J \tW \cap \Adm^J(\mu)={}^J \tW \cap \Adm(\mu)$. 
\end{thm}

\begin{proof}
Let $\Phi$ be the relative root system and $\Phi_a$ be the affine root system, which is a set of affine functions on $V=X_*(S) \otimes \RR$ of the form $\b+r$ for $\b \in \Phi$ and $r \in \RR$. 

Let $w \in {}^J \tW \cap \Adm^J(\mu)$. Then $w \le \max(W_J t^\g W_J)$ for some $\g \in W_0 \cdot \underline \mu$. 

We first show that 

(a) $w \le \max(t^\l W_J)$ for some $\l \in W_0 \cdot \underline \mu$ with $t^\l \in {}^J \tW$. 

For $y \in {}^J \tW$, we set $I(J, y)=\max\{K \subset J; y(K)=K\}$. By \cite[Corollary 2.6]{HeMin}, $t^\g$ is conjugate by an element in $W_J$ to an element $z=x w_1$, where $w_1 \in {}^J \tW$ and $x \in W_{I(J, w_1)}$. Since $z$ is conjugate to $t^\g$, it is of the form $t^{\l}$ for some $\l \in W_0 \cdot \underline \mu$. 

Let $\Phi_1$ be the root system associated to $I(J, w_1)$. By definition, for any $\a \in \Phi_1$, $t^\l(\a) \in \Phi_1$. Therefore $t^{\l}(\a)-\a=\<\l, \a\>$ is in the root lattice of $\Phi_1$. However, any nonzero $r \in \Phi_a$ is not spanned by $K$ for any $K \subset \tSS$ with $W_K$ finite. Hence $\<\l, \a\>=0$ and $t^\l(\a)=\a$ for all $\a \in \Phi_1$. In particular, $t^\l \in {}^{I(J, w_1)} \tW$. Since $w_1 \in {}^{I(J, w_1)} \tW$ and $t^\l \in W_{I(J, w_1)} w_1$, we must have $x=1$. 

(a) is proved. 

We may write $\max(t^\l W_J)$ as $a b$, where $a \in W_J$ and $b \in {}^J \tW$. Since $t^\l \in {}^J \tW$, $b=t^\l y$ for some $y \in W_J$ with $\ell(t^\l y)=\ell(t^\l)+\ell(y)$. If $y \neq 1$, then $s_i y<y$ for some $i \in J$. Let $\a_i$ be the simple root associated to $s_i$. Since $t^\l \in {}^J \tW$, $t^{-\l} \in \tW^J$. Hence $t^{-\l} (\a_i)=\a_i-\<\l, \a\>$ is a positive affine root. 
Hence $\<\l, \a_i\> \le 0$. 

If $\<\l, \a_i\><0$, then $t^\l(\a_i)$ is a negative affine root and $t^\l s_i<t^\l$, which contradicts the fact that  $\ell(t^\l y)=\ell(t^\l)+\ell(y)$. If $\<\l, \a_i\>=0$, then $t^\l y=s_i t^\l (s_i y)$, which contradicts the fact that $t^\l y \in {}^J \tW$. Therefore $y=1$ and $w \le t^\l$. 
\end{proof}

\subsection{} We prove Theorem \ref{main} (2). 

Let $J \subset \tSS$ such that $\s(J)=J$ and $W_J$ is finite. Recall that $Y_{J, \mu}=\cup_{w \in \Adm(\mu)} P_J w P_J$. By \S\ref{6-3} and Theorem \ref{comp},
\begin{align*} Y_{J, \mu} &=\cup_{x \in \Adm^J(\mu) \cap {}^J \tW} P_J \cdot_\s I x I \\ &=\cup_{x \in \Adm(\mu) \cap \tW} P_J \cdot_\s I x I \\ & \subset P_J \cdot_\s Y_{\emptyset, \mu}.\end{align*} Therefore, \[\tag{a} Y_{J, \mu}=P_J \cdot_\s Y_{\emptyset, \mu}.\]

For any $J \subset J' \subset \tSS$ with $\s(J)=J$, $\s(J')=J'$ and $W_{J'}$ finite, we have
\[\tag{b} Y_{J', \mu}=P_{J'} \cdot_\s Y_{\emptyset, \mu}=P_{J'} \cdot_\s (P_J \cdot_\s Y_{\emptyset, \mu})=P_{J'} \cdot_\s Y_{J, \mu}.\]

Now $$\pi_{J, J'} (X(\mu, b)_J)=\{g P_{J'} \in G(L)/P_{J'}; g \i b \s(g) \in P_{J'} \cdot_\s Y_{J, \mu}\}=X(\mu, b)_{J'}.$$ In other words, $\pi_{J, J'}$ is surjective. 

\subsection{} Now we prove Mazur's inequality for $J$. 

%The direction $b \in B(G, \mu) \Rightarrow X(\mu, b)_J \neq \emptyset$ is proved in Theorem \ref{converse-Mazur}. 

If $X(\mu, b)_J \neq \emptyset$, then by Theorem \ref{main} (2), $X(\mu, b)_\emptyset \neq \emptyset$. By Theorem \ref{Mazur}, $b \in B(G, \mu)$. 

\appendix

\section{Admissibility}

\subsection{}\label{In} In the appendix, we assume that $F=\FF_q((\e))$. We first recall the Moy-Prasad filtration \cite{MP}. 

Let $v$ be a generic point in the base alcove $\aa$. For any $r \ge 0$, let $I_r$ be the subgroup of $I$ generated by a suitable subgroup of $T(L)_1$ and $U_\phi$, where $\phi$ runs over all the affine roots with $\phi(v) \ge r$. By definition, if $x \in \tW$ with $\ell(x)<r$, then $h I_r h \i \subset I_{r-\ell(x)} \subset I$ for any $h \in I x I$. 

\subsection{} A subset $V$ of $G(L)$ is called {\it admissible} if for any $w \in \tW$, there exists $r \ge 0$ such that $\cup_{w' \le w} (V \cap I w' I)$ is stable under the right action of $I_r$. This is equivalent to say that for any $w \in \tW$, there exists $r' \ge 0$ such that $V \cap I w I$ is stable under the right action of $I_{r'}$. 

An admissible subset $V$ of $G(L)$ is a locally closed subscheme if for any $w \in \tW$ and $r \ge 0$ such that $\cup_{w' \le w} (V \cap I w' I)$ is stable under the right action of $I_r$, $\cup_{w' \le w} (V \cap I w' I)/I_r$ is a locally closed subscheme of $\overline{I w I/I_r}=\cup_{w' \le w} I w' I/I_r \subset G(L)/I_r$. 

We define the closure of a locally closed subscheme $V$ in $G$ as follows. 

Let $w \in \tW$. Let $r \ge 0$ such that $\cup_{w' \le w} (V \cap I w' I)$ is stable under the right action of $I_r$. Let $V_w$ be the inverse image under the projection $G(L) \to G(L)/I_r$ of the closure of $\cup_{w' \le w} (V \cap I w' I)/I_r$ in $G(L)/I_r$. Then it is easy to see that $V_w$ is independent of the choice of $r$. Moreover, if $w' \le w$, then $V_{w'} \subset V_w$. Set $$\overline{V}=\varinjlim_w V_w.$$

\begin{thm}\label{adm}
Let $[b]$ be a $\s$-conjugacy class of $G(L)$. Then $[b]$ is admissible. 
\end{thm}

\begin{rem}
For split groups, this is first proved by Hartl and Viehmann in \cite{HV}. 
\end{rem}

\begin{proof}
Let $w$ be a $\s$-straight element in $[b]$. By \S \ref{g-stable} (3), $[b]=G(L) \cdot_\s I w I$. Let $y \in \tW$ such that $[b] \cap I y I \neq \emptyset$. By \cite[Theorem 1.4]{RZ}, there exists $n \in \NN$ such that for any $g \in [b] \cap I y I$, $h \i g \s(h) \in I w I$ for some $h \in I z I$ with $\ell(z)<n$. By \S \ref{In}, $h \i g I_n \s(h) \subset h \i g \s(h) I_{n-\ell(z)} \subset I w I$. Hence $g I_n \subset G(L) \cdot_\s I w I=[b]$. The theorem is proved. 
\end{proof}

\

Another admissibility result we need is the following:

\begin{thm}\label{w-closure'}
Let $w \in \tW$. Then $G(L) \cdot_\s I w I$ is admissible and $$\overline{G(L) \cdot_\s I w I}=\cup_{w' \le w} G(L) \cdot_\s I w' I.$$ 
\end{thm}

\begin{proof}
Set $V=G(L) \cdot_\s I w I$ and $V'=\cup_{w' \le w} G(L) \cdot_\s I w' I$. By Theorem \ref{w-closure}, both $V$ and $V'$ are finite unions of $\s$-conjugacy classes and $V'=\sqcup_{\co \preceq_\s w} [\co]$. By Theorem \ref{adm}, $V$ and $V'$ are admissible. 

Let $x \in \tW$. By \cite[Theorem 1.4]{RZ}, there exists $n \in \NN$ such that 
\begin{gather*}
V \cap \overline{I x I}=(\cup_{z \in \tW, \ell(z)<n} I z I) \cdot_\s I w I \cap \overline{I x I}; \\
V' \cap \overline{I x I}=(\cup_{z \in \tW, \ell(z)<n} I z I) \cdot_\s \overline{I w I} \cap \overline{I x I}.
\end{gather*}

Define the action of $I_n$ on $(\cup_{z \in \tW, \ell(z)<n} I z I) \times G(L)/I_n$ by $h \cdot (g, g')=(g h \i, h g')$. We denote by $(\cup_{z \in \tW, \ell(z)<n} I z I) \times^{I_n} G(L)/I_n$ its quotient. Consider the map $$(\cup_{z \in \tW, \ell(z)<n} I z I) \times G(L)/I_n \to G(L)/I, \qquad (g, g') \mapsto g g' \s(g) \i.$$By \S \ref{In}, it is well-defined. It induces a map $$\pi: (\cup_{z \in \tW, \ell(z)<n} I z I) \times^{I_n} \overline{I w I}/I_n \to G(L)/I.$$ This is a proper map. Hence the image is closed in $G(L)/I$ and is the closure of the image of $(\cup_{z \in \tW, \ell(z)<n} I z I) \times^{I_n} I w I/I_n$. 

Therefore $V' \cap \overline{I x I}$ is closed and is the closure of $V \cap \overline{I x I}$. In other words, $V_x=V' \cap \overline{I x I}$. Hence $$\overline{V}=\varinjlim_x V_x=V'.$$
\end{proof}

\section*{Acknowledgment} We thank T. Haines, M. Rapoport and X. Zhu for useful discussions.

\end{document}